\RequirePackage{xcolor}
\documentclass[journal,twoside,web]{cls/ieeecolor}
\usepackage{cls/lcsys}

\let\labelindent\relax

\usepackage[noadjust]{cite}
\usepackage{enumitem}
\usepackage{amsmath,amssymb,amsfonts,amsthm,dsfont} 
\usepackage{mathtools}
\mathtoolsset{showonlyrefs}

\usepackage{algorithm}
\usepackage[noend]{algpseudocode}
\algnewcommand{\Output}{\item[\algorithmicoutput]}
\algnewcommand{\algorithmicoutput}{\textbf{Output:}}

\usepackage{verbatim}
\usepackage[normalem]{ulem}
\usepackage{bibunits}

\usepackage{graphicx}
\graphicspath{{fig/}}
\usepackage{caption}
\usepackage{subcaption}

\usepackage[breaklinks=true, colorlinks, bookmarks=true]{hyperref}

\newtheorem{theorem}{Theorem}[section]
\newtheorem{lemma}[theorem]{Lemma}

\newtheorem{proposition}[theorem]{Proposition} 
\theoremstyle{definition}

\newtheorem{assumption}[theorem]{Assumption}
\newtheorem{problem}{Problem}
\theoremstyle{remark}

\newcommand{\oprocendsymbol}{\hbox{$\bullet$}}
\newcommand{\oprocend}{\relax\ifmmode\else\unskip\hfill\fi\oprocendsymbol}

\newcommand{\ba}{\begin{align}}
\newcommand{\ea}{\end{align}}

\DeclareMathOperator*{\argmin}{argmin}

\renewcommand{\d}[1]{\ensuremath{\operatorname{d}\!{#1}}}

\newcommand{\calC}{{\cal C}}

\newcommand{\calI}{{\cal I}}

\newcommand{\calL}{{\cal L}}

\newcommand{\calN}{{\cal N}}

\newcommand{\calP}{{\cal P}}

\newcommand{\calS}{{\cal S}}

\newcommand{\calX}{{\cal X}}

\newcommand{\bfd}{\mathbf{d}}

\newcommand{\bff}{\mathbf{f}}
\newcommand{\bfg}{\mathbf{g}}
\newcommand{\bfh}{\mathbf{h}}

\newcommand{\bfn}{\mathbf{n}}

\newcommand{\bfu}{\mathbf{u}}

\newcommand{\bfx}{\mathbf{x}}
\newcommand{\bfy}{\mathbf{y}}
\newcommand{\bfz}{\mathbf{z}}

\newcommand{\bfphi}{\boldsymbol{\phi}}

\newcommand{\bfxi}{\boldsymbol{\xi}}

\newcommand{\bbI}{\mathbb{I}}

\newcommand{\bbR}{\mathbb{R}}

\title{\LARGE \bf Constrained Variational Inference via Safe Particle Flow }

\author{Yinzhuang Yi \quad Jorge Cort{\'e}s \quad Nikolay Atanasov {}
\thanks{We gratefully acknowledge support from ONR Award N00014-23-1-2353 and NSF FRR CAREER 2045945.}%
\thanks{The authors are with the Contextual Robotics Institute, University of California San Diego, La Jolla, CA 92093, USA (e-mails: {\tt \{yiyi,cortes,natanasov\}@ucsd.edu}).}%
}

\begin{document}
	\maketitle
	\thispagestyle{empty}
	\begin{abstract}
	We propose a control barrier function (CBF) formulation for enforcing equality and inequality constraints in variational inference. The key idea is to define a barrier functional on the space of probability density functions that encode the desired constraints imposed on the variational density. By leveraging the Liouville equation, we establish a connection between the time derivative of the variational density and the particle drift, which enables the systematic construction of corresponding CBFs associated to the particle drift. Enforcing these CBFs gives rise to the safe particle flow and ensures that the variational density satisfies the original constraints imposed by the barrier functional. This formulation provides a principled and computationally tractable solution to constrained variational inference, with theoretical guarantees of constraint satisfaction. The effectiveness of the method is demonstrated through numerical simulations.
	\end{abstract}

    \begin{IEEEkeywords}
    Estimation, Variational methods, Safety-critical control
    \end{IEEEkeywords}

	\section{Introduction}
\IEEEPARstart{B}{ayesian} inference plays a key role in a variety of applications, including statistical learning \cite{khan2023bayesian}, estimation theory \cite{barfoot2017state}, and motion planning \cite{mukadam2018continuous}. In Bayesian inference problems, we start with a prior probability density function (PDF) $p(\bfx)$, and given an observation $\bfz$ and associated likelihood PDF $p(\bfz | \bfx)$, we aim to compute a posterior PDF $p(\bfx|\bfz)$, following Bayes' rule. Traditional approaches to Bayesian inference include the Kalman filter \cite{kalman1960new}, which relies on linear Gaussian assumptions, and its extension to nonlinear observation and transition models, the extended Kalman filter (EKF) \cite{anderson2005optimal}. As an alternative, particle filters \cite{gordon1993novel} and sequential Monte Carlo methods \cite{neal2011mcmc} approximate the posterior using weighted particles instead of a parametric density.

Variational inference (VI) \cite{jordan1999introduction} is a formulation of Bayesian inference as an optimization problem with Kullback–Leibler (KL) divergence between a variational density $q(\bfx)$ and the posterior density $p(\bfx | \bfz)$ as the objective. Many effective methods~\cite{bishop2006pattern} exist for VI. A particularly effective approach is particle-based VI \cite{liu2016stein, daum2007nonlinear, geffner2023langevin}, where the variational density is approximated by a finite set of particles that evolve according to a particle drift function. Examples of such methods include the Stein variational gradient descent \cite{liu2016stein}, the particle flow particle filter \cite{daum2007nonlinear}, and diffusion-based VI \cite{geffner2023langevin}.

In this paper, we consider a constrained VI problem where the posterior density must satisfy certain conditions, e.g., manifold constraints on orientation in robot state estimation \cite{lee2019space}. Constrained VI can be approached by requiring particle samples from the variational density to satisfy the constraints. Methods using this approach typically modify the particle drift to ensure constraint satisfaction \cite{heiden2022probabilistic, zhang2022sampling, power2024constrained}. For instance, equality constraints have been addressed using projection methods \cite{zhang2022sampling} and Lagrange multiplier formulations \cite{heiden2022probabilistic}. The projection method \cite{zhang2022sampling} is extended in \cite{power2024constrained} to handle multiple equality and inequality constraints. However, the inclusion of inequality constraints relies on slack variables, which leads to poor constraint satisfaction in practice. Instead of altering an established particle drift, one can also derive particle drift using a penalized objective that encodes the prescribed constraints \cite{tabor2025csvd, gurbuzbalaban2024penalized}. The penalized objective can be constructed either by augmenting the KL divergence with an additional penalty term \cite{tabor2025csvd} or by modifying the posterior density \cite{tabor2025csvd, gurbuzbalaban2024penalized}. A hybrid approach has been proposed in \cite{craft2024nonlinear}, where a particle drift is derived from a modified posterior that encodes inequality constraints, and subsequently adjusted to enforce equality constraints through a projection method similar to \cite{zhang2022sampling}. Nonetheless, penalization methods lack guarantees on exact constraint satisfaction. Moreover, all aforementioned methods enforce constraints only on individual particles and do not formally establish whether particle-wise constraint satisfaction guarantees variational-density-wise constraint satisfaction. To address this limitation, Chamon et al. \cite{chamon2024constrained} formulate a primal-dual approach, where the variational density and the dual variables associated with the constraints are updated simultaneously following the steepest descent/ascent direction of the Lagrangian under the Wasserstein metric. This method provides theoretical guarantees on the variational density asymptotically satisfying chance constraints on the expectations of the constraint functions, which is different from the deterministic constraints on the density function support studied here.

Our objective is to preserve the simplicity of modifying a desired particle drift while providing guarantees on constraint satisfaction for the variational density.  This motivates using control barrier functions (CBFs), which provide a rigorous approach to enforce constraints on the evolution of a control system~\cite{ames2017clfcbfqp, jorege2024flow}.  The safe particle flow introduced here can be interpreted as a minimally modified instance of the gradient flow of the KL divergence in the space of PDFs.  Recent developments in the CBF literature introduce the safe gradient flow~\cite{jorege2024flow}, demonstrating the effectiveness of CBFs for enforcing constraints along the gradient flow dynamics in nonlinear optimization.  CBF techniques have gained popularity in the control community due to their simplicity and formal guarantees for constraint satisfaction.  For control-affine systems, CBF conditions take the form of linear constraints in the control input, enabling safe control synthesis via quadratic programming. A comprehensive overview of CBF techniques and their application as safety constraints in quadratic programs is provided in \cite{ames2017clfcbfqp}. These methods have been extended to functional spaces, illustrating the use of control barrier functionals for time-delayed system safety \cite{kiss2023cbfal}. Despite their success in enforcing constraints within control systems, CBF techniques have not been used to enforce constraints in VI.

We propose a methodology that leverages CBF techniques to incorporate inequality constraints in VI in three key steps: (1) a barrier functional encodes the desired constraints on the variational density, yielding conditions on its time derivative. These are challenging to work with directly due to the infinite-dimensional nature of the density space; (2) the Liouville equation \cite{wibisono2017information} relates the time derivative of the variational density to the deterministic particle drift, allowing us to instead formulate corresponding CBF constraints on the particle drift; (3) satisfying the CBF constraints on the deterministic particle drift guarantees that the variational density flow also satisfies the barrier constraints. The deterministic particle drift is essential for the validity of CBF techniques. The CBF-based flow formulation offers an anytime algorithm that ensures constraint satisfaction at any time during the particle evolution, unlike Lagrange multiplier–based methods that only guarantee asymptotic constraint satisfaction. Our method provides a principled and computationally tractable way of ensuring density-wise constraint satisfaction, thereby establishing a rigorous connection between safe particle flow and constrained VI.

    \section{Problem Statement}\label{sec: prob}
Consider a Bayesian inference problem where $\bfx \in \calX \subset \bbR^{n}$ is a random variable of interest with prior PDF $p(\bfx)$. Given a measurement $\bfz \in \bbR^m$ with likelihood PDF $p(\bfz | \bfx)$, the posterior PDF of $\bfx$ conditioned on $\bfz$ is determined by Bayes' theorem \cite{bayes1763}: 
\begin{equation}
\label{prelim: bayes_rule}
    p(\bfx | \bfz) = \frac{p(\bfz | \bfx) p(\bfx)}{p(\bfz)},
\end{equation}
where $p(\bfz)$ is the marginal measurement PDF, computed as $p(\bfz) = \int_{\calX} p(\bfz | \bfx) p(\bfx) \d \bfx$. Computing the posterior PDF in \eqref{prelim: bayes_rule} is often intractable because $p(\bfz)$ may not have a closed form, except when the prior $p(\bfx)$ and the likelihood $p(\bfz | \bfx)$ are a conjugate pair. To calculate the posterior of non-conjugate prior and likelihood, approximation methods are needed.

This problem can be approached using VI methods \cite[Ch.~10]{bishop2006pattern}, which approximate the posterior $p(\bfx | \bfz)$ in \eqref{prelim: bayes_rule} using a PDF $q(\bfx)$ with tractable expression, termed variational density. To perform the approximation, VI minimizes the KL divergence between the true posterior $p(\bfx | \bfz)$ and the variational density $q(\bfx)$:
\begin{equation}
\label{prelim: kld}
    D_{KL} \big( q(\bfx) \| p(\bfx | \bfz) \big) = \int_{\calX} q(\bfx) \log \bigg(\frac{q(\bfx)}{p(\bfx | \bfz)} \bigg) \d\bfx.
\end{equation}
The variational density $q(\bfx)$ can be represented as a collection of weighted particles \cite{liu2016stein}, leading to the particle-based VI considered in this paper.

This paper considers a constrained VI problem with \textit{support constraints}: the variational density is restricted to take nonzero values only within a designated set, which we term safe. The set $\calS$ is defined as the intersection of $N$ inequality-constrained sets, each specified implicitly by a continuously differentiable function $g_i: \calX \to \bbR$:
\begin{equation}
    \calS = \bigcap_{i \in \calI} \calS_i, \quad \calS_i = \{ \bfx \in \calX \;|\; g_i(\bfx) \geq 0 \} ,
\end{equation}
where $\calI = \{1, \dots, N\}$. Equality constraints specified by $g^e(\bfx)$ can be equivalently represented by a pair of inequality constraints $g^e(\bfx) = 0 \Leftrightarrow [g^e(\bfx), - g^e(\bfx)]^{\top} \geq \bf0$. The safe set can encode, for instance, geometric constraints, such as manifold constraints in robot state estimation. We want to find a variational density that matches the Bayes' posterior as much as possible while satisfying the safety constraint $q(\bfx) = 0$, for all $\bfx \in \calX \setminus \calS$. This constraint is essential to ensure that the density is supported on the safe set. We make the following assumptions throughout the paper.
\begin{assumption}[Feasibility]
\label{assum: state_safe_separation}
    The state space $\calX \subset \bbR^n$ is bounded and the safe set $\calS$ is nonempty.
\end{assumption}
The above assumption ensures the existence of a feasible solution to the constrained VI problem. In addition, we make the following assumption on the variational densities.

\begin{assumption}[Variational Density Family]
\label{assum: valide_densities}
  The variational density $q(\bfx)$ belongs to the family $\calP = \{ p(\bfx) \in \calL^1(\calX) \mid  \int_{\calX} p(\bfx) \d \bfx = 1, \; p(\bfx) \geq 0 \}$, where $\calL^1(\calX)$ is the space of absolutely integrable functions on $\calX$ with respect to the Lebesgue measure.
\end{assumption}

Assumption~\ref{assum: valide_densities} ensures that the variational densities are valid PDFs, supported on the state space $\calX$. We formulate the constrained VI problem as follows.

\begin{problem}\label{pr:opt}
    Find a variational density $q(\bfx)$ that solves the optimization problem:
    \begin{equation}
    \label{prob: cons_bayes_infer_functional}
	   \begin{aligned}
		\min_{q(\bfx) \in \calP} \, &\, D_{KL} \big( q(\bfx) \| p(\bfx | \bfz) \big)  \\
		\mathrm{s.t.} \, &\, \int_{\calX \setminus \calS} q(\bfx) \d \bfx = 0. 
	\end{aligned}
    \end{equation}
    The constraint ensures that $q(\bfx)$ has support strictly on $\calS$.
\end{problem}

    \section{Variational Inference Using Particle Flow}

Before considering the constrained formulation in Problem~\ref{pr:opt}, we review a gradient flow method for solving unconstrained VI problems \cite{chen2023gradient}. In this approach, an initial guess for the variational density is modified by following the steepest descent direction of the KL divergence functional. This yields a continuous-time trajectory $q(\bfx;t)$ whose asymptotic limit is a solution to the VI problem.

\subsection{Gradient Flow} \label{sec: gradient_flow}
The tangent space $T_q\calP$ of the density family $\calP$ at a PDF $q \in \calP$ is~\cite{chen2023gradient}:
\begin{equation}
    T_q \calP = \Big \{ \sigma(\bfx) \in \calL^{1}(\calX) \mid \int \sigma(\bfx) \d \bfx = 0 \Big \}. 
\end{equation}
The cotangent space $T^*_q\calP$ is the dual of $T_{q}\calP$. We can introduce a bilinear map $\langle \cdot, \cdot \rangle_{\calP}$ as the duality pairing $T^*_q\calP \times T_{q}\calP \to \bbR$. For any $\psi \in T^*_q\calP$ and $\sigma \in T_{q}\calP$, the pairing can be identified as $\langle \psi, \sigma \rangle_{\calP} = \int_{\calX} \psi(\bfx) \sigma(\bfx) \d \bfx$. The first variation of the KL divergence $\frac{\delta D_{KL}(q || p)}{\delta q}$ with respect to $q \in \calP$ is an element of the cotangent space $T^*_q\calP$, given by $\frac{\delta D_{KL}(q || p)}{\delta q} = 1 + \log \frac{q}{p}$. Given a metric tensor at $q$, denoted by $M(q) : T_{q}\calP \to T^*_q\calP$, we can express the Riemannian metric $g_q : T_{q}\calP \times T_{q}\calP \to \bbR$ as $g_q(\sigma_1, \sigma_2) = \langle M(q)\sigma_1, \sigma_2 \rangle_{\calP}$. The gradient of the KL divergence under the Riemannian metric, denoted by $\nabla_{q}D_{KL}(q || p)$, is defined as: $g_q(\nabla_{q}D_{KL}(q || p), \sigma) = \bigl \langle \frac{\delta D_{KL}(q || p)}{\delta q}, \sigma \bigr \rangle_{\calP}$, for any $\sigma \in T_{q}\calP$. Using the metric tensor, we can write $\nabla_{q}D_{KL}(q || p) = M^{-1}(q) \frac{\delta D_{KL}(q || p)}{\delta q}$. The gradient flow of the KL divergence with respect to this metric is:
\begin{equation}\label{prelim: gradient_flow}
    \frac{\partial q(\bfx; t)}{\partial t} =  - \nabla_{q}D_{KL}(q || p) \big|_{q = q(\bfx; t)}.
\end{equation}
The convergence of the gradient flow to the optimum depends on the choice of Riemannian metric. For a detailed analysis of convergence properties under different metrics, we refer the reader to~\cite{chen2023gradient}. However, implementing the gradient flow directly is challenging because it is defined over the infinite-dimensional space of PDFs. To address this, the variational density $q(\bfx;t)$ can be represented by a finite set of samples drawn from it, referred to as \emph{particles}, and its evolution can be characterized through the evolution of these particles. The connection between the particle evolution and the gradient flow dynamics is formalized by the Liouville equation \cite{wibisono2017information}, which we review next.

\subsection{Liouville Equation and Particle Flow}\label{sec: liouville_pf}
Consider a random process $\bfx(t) \in \bbR^n$ governed by the ordinary differential equation (ODE):
\begin{equation} \label{prelim: particle_dynamics}
    \frac{\d \bfx(t)}{\d t} = \bfphi(\bfx(t), t),
\end{equation}
where $\bfphi(\bfx(t), t)$ is the particle drift. Then, the PDF $q(\bfx; t)$ of $\bfx(t)$ evolves according to the Liouville equation \cite{wibisono2017information}: 
\begin{equation}
\label{prelim: liouville}
	\frac{\partial q(\bfx; t)}{\partial t} = - \nabla_{\bfx} \cdot \big( q(\bfx; t) \bfphi(\bfx, t) \big), 
\end{equation}
where $\nabla_{\bfx} \cdot $ denotes the divergence operator. The Liouville equation \eqref{prelim: liouville} establishes a mapping between the functional gradient in density space \eqref{prelim: gradient_flow} and the particle drift $\bfphi(\bfx(t), t)$. Since the state space $\calX$ considered in this paper is a subset of $\bbR^n$, we impose the following assumption to ensure that the Liouville equation holds in our setting.

\begin{assumption}[Conservation of Probability Mass] \label{assum: mass_conservation}
    On the boundary of the state space, the particle drift $\bfphi(\bfx, t)$ satisfies $\langle \bfphi(\bfx,t), \hat{\bfn}(\bfx) \rangle_{\bbR^n} = 0$, for all $\bfx \in \partial \calX$, where $\langle \cdot, \cdot \rangle_{\bbR^n}$ denotes the standard Euclidean inner product on $\bbR^n$ and $\hat{\bfn}(\bfx)$ is the outward unit normal vector to  $\partial \calX$ at $\bfx$.
\end{assumption}

This assumption ensures~\cite{gardiner2009stochastic} that the trajectories $p(\bfx; t)$ governed by the Liouville equation \eqref{prelim: liouville} remain in $\calP$. Note that the particle drift must remain deterministic given our reliance on CBF techniques to ensure safety.

Now, consider a particle $\bfx(t)$ sampled from the variational density $q(\bfx; t)$, whose evolution follows the gradient flow dynamics in \eqref{prelim: gradient_flow}. The Liouville equation \eqref{prelim: liouville} is used to derive a corresponding particle drift $\bfphi(\bfx, t)$ that induces the desired gradient flow for the variational density $q(\bfx; t)$, by solving:
\begin{equation} \label{sspf: particle_flow}
    \nabla_{q}D_{KL}(q || p) \big|_{q = q(\bfx; t)} = \nabla_{\bfx} \cdot \big( q(\bfx; t) \bfphi(\bfx, t) \big).
\end{equation}
The evolution of the particle is then governed by \eqref{prelim: particle_dynamics}, referred to as the \emph{particle flow}. In this paper, we focus on the particle flow derived using the Stein Riemannian metric, originally introduced in \cite{liu2016stein}. Specifically, we consider the inverse metric tensor satisfying \cite{chen2023gradient}
\begin{equation}
    M^{-1}(q) \psi = - \nabla_{\bfx} \cdot \left( q(\bfx) \int_{\calX} k(\bfx, \bfxi) q(\bfxi) \nabla_{\bfxi} \psi(\bfxi) \d \bfxi \right), 
\end{equation}
where $\psi \in T^*_{q}\calP$ with $k(\cdot, \cdot)$ denotes a positive definite kernel. The particle drift is obtained by solving \eqref{sspf: particle_flow}:
\begin{align}
    &\bfphi_{d}(\bfx, t) = - \int_{\calX} k(\bfx, \bfxi) q(\bfxi; t) \nabla_{\bfxi} \log \left( \frac{q(\bfxi; t)}{p(\bfxi, \bfz)} \right) \d \bfxi \notag\\
    &= \int_{\calX} q(\bfxi; t) \left( \nabla_{\bfxi} k(\bfx, \bfxi) + k(\bfx, \bfxi) \nabla_{\bfxi} \log p(\bfxi, \bfz) \right) \d \bfxi \label{sspf: stein_flow}\\
    &\approx \frac{1}{M} \sum_{j=1}^{M} \Big( \nabla_{\bfxi} k(\bfxi, \bfx) + k(\bfxi, \bfx) \nabla_{\bfxi} \log p(\bfxi, \bfz) \Big) \bigg|_{\bfxi = \bfx_j(t)}, \notag
\end{align}
where $\{ \bfx_j(t) \}_{j=1}^M \sim q(\bfx; t)$ are particles sampled from the variational density at time $t$. The second equality follows from integration by parts, while the last step uses Monte Carlo integration to approximate the expectation.

    \section{Safe Particle Flow}
The gradient flow of the KL divergence in~\eqref{prelim: gradient_flow} does not take the constraint in the optimization \eqref{prob: cons_bayes_infer_functional} into consideration. Directly modifying the gradient flow dynamics is challenging because of the infinite-dimensional nature of the space of PDFs and the difficulty of enforcing the constraint.

Our approach to deal with Problem~\ref{pr:opt} exploits the connection between the gradient flow dynamics and the particle drift established by the Liouville equation \eqref{sspf: particle_flow}. In the forthcoming discussion, we first introduce a barrier functional in the space of PDFs and use it to formulate constraints on the variational density flow so that its continuous-time trajectory satisfies the constraint in \eqref{prelim: gradient_flow} at all times. Using the Liouville equation, we then translate the constraints imposed on the variational density into equivalent constraints on the particle drift. Finally, we construct a safe particle drift by modifying a desired drift obtained from the unconstrained VI problem.

\subsection{Barrier Functions}
We first review barrier functions in finite-dimensional vector spaces~\cite{blanchini1999set_invariance, prajna2004safety}. Consider an autonomous system $\dot{\bfy} = \bff(\bfy)$ in $\bbR^n$ with state trajectories denoted by $\bfy(t)$. The safety of the system can be certified by ensuring that the trajectories remain within a safe set $\calC \subset \bbR^n$. This is equivalent to showing that $\calC$ is forward-invariant~\cite{blanchini1999set_invariance}. 

To establish forward invariance, we introduce a continuously differentiable function $b: \bbR^n \to \bbR$ that encodes the safe set as its zero-superlevel set, $\calC = \{ \bfy \in \bbR^n \mid b(\bfy) \geq 0 \}$. Such a function is called a \emph{barrier function}. A sufficient condition for the forward invariance of $\calC$ is that the barrier function satisfies a differential inequality along system trajectories: $\frac{\d b(\bfy(t))}{\d t} + \alpha_b b(\bfx) \geq 0$, where $\alpha_b > 0$ is a positive constant. This ensures forward invariance of the safe set \cite[Theorem~3.1]{blanchini1999set_invariance}.

\subsection{Barrier Functional Construction}
Since the variational density evolves in the space of PDFs, rather than in a finite-dimensional vector space, we rely on the concept of \emph{barrier functional}~\cite{kiss2023cbfal} to extend the forward invariance property to PDF trajectories. Based on the constraint in \eqref{prob: cons_bayes_infer_functional}, the set of feasible densities is:
\begin{equation}
\label{prob: variational_safe_set}
    \calP_{s} = \{ p(\bfx) \in \calP \;|\; \int_{\calX \setminus \calS} p(\bfx) \d \bfx = 0 \}.
\end{equation}
We define the barrier functional $\bfh: \calP \to \mathbb{R}^N$, whose $i$th component is given by
\begin{equation}
\label{sspf: variational_cbf}
    h_i(q(\bfx; t)) = -\int_{\calX \setminus \calS_i} g_i(\bfx) q(\bfx; t) \d \bfx, \ i \in \calI.
\end{equation}
We show next that the zero-level set of the barrier functional coincides with the set of feasible densities defined in~\eqref{prob: variational_safe_set}.
\begin{lemma}[Consistent Barrier Functional]\label{lem: consis_bfal}
    The zero-level set of the barrier functional introduced in \eqref{sspf: variational_cbf} satisfies $\{ p(\bfx) \in \calP \;|\; \bfh(p(\bfx)) = \mathbf{0} \} = \calP_s$.
\end{lemma}
\begin{proof}
    Based on the definition of the constraint set $\calS_i$, $g_i(\bfx) < 0$ for all $\bfx \in \calX \setminus \calS_i$. Consequently, $-g_i(\bfx) p(\bfx) \geq 0$ for all $\bfx \in \calX \setminus \calS_i$. Given the definition in~\eqref{sspf: variational_cbf}, the condition $h_i(p(\bfx)) = 0$ implies that $p(\bfx) = 0$ for all $\bfx \in \calX \setminus \calS_i$. Therefore, if $\bfh(p(\bfx)) = \mathbf{0}$, then $p(\bfx) = 0$ for all $\bfx \in \bigcup_{i \in \calI} \left( \calX \setminus \calS_i \right) = \calX \setminus \bigcap_{i \in \calI} \calS_i = \calX \setminus \calS$. Hence, we obtain $\int_{\calX \setminus \calS} p(\bfx)\, \d \bfx = 0$.
\end{proof}

Based on Lemma~\ref{lem: consis_bfal}, ensuring forward invariance of $\calP_s$ is equivalent to ensuring $\bfh(p) = \mathbf{0}$. Since $\bfh(p) \geq \mathbf{0}$ for all $p \in \calP$ by definition, this is equivalent to ensuring $\bfh(p) \leq \mathbf{0}$. This yields the following barrier constraint: 
\begin{equation}
\label{sspf: variational_cbc}
    \frac{\d \bfh(q(\bfx; t))}{\d t} + \alpha_h \bfh(q(\bfx; t)) \leq \mathbf{0},
\end{equation}
where $\alpha_h > 0$ is a positive design parameter that determines the convergence rate of $q(\bfx; t)$ to the safe set boundary \cite{ames2019cbf}.

\subsection{Safe Particle Flow}
The constraint \eqref{sspf: variational_cbc} can be used to derive corresponding constraints on the particle drift $\bfphi(\bfx, t)$, hence, establishing necessary conditions for the particle drift to render the feasible density set $\calP_s$ forward invariant. 

As we show next, the conditions on the particle drift $\bfphi(\bfx, t)$ can be expressed in terms of its Euclidean inner product with the gradient of the constraint functions $\nabla_{\bfx} g_i(\bfx)$. This result facilitates our construction of a safe particle drift in the next section using CBF techniques.

\begin{theorem}[Safe Particle Flow]
    Let $\bfphi(\bfx, t)$ be a particle drift satisfying, for all $i \in \calI$,
    \begin{align} \label{sspf: dynamics_cbc}
        \int_{\calX \setminus \calS_i} \! \! \! \! q(\bfx; t) \langle \nabla_{\bfx}g_i(\bfx), \bfphi(\bfx, t) \rangle_{\bbR^n} \! \d \bfx \! \geq \! \alpha_h h_i(q(\bfx; t)).\;
    \end{align}
    Under Assumptions~\ref{assum: state_safe_separation}, \ref{assum: valide_densities} and~\ref{assum: mass_conservation}, the particle flow \eqref{prelim: particle_dynamics} ensures that the feasible density set $\calP_s$ is forward-invariant and exponentially stable.
\end{theorem}

\begin{proof}
    Assumption~\ref{assum: state_safe_separation} ensures that the feasible density set is nonempty. Under the regularity conditions stated in Assumption~\ref{assum: valide_densities}, the time derivative of the $i$th component of the barrier functional is given by
    \begin{equation}
    \label{sspf: time_derivative_variational_cbf}
        \frac{\d h_i(q(\bfx; t))}{\d t} = \bigl \langle \frac{\delta h_i(q(\bfx; t))}{\delta q(\bfx; t)}, \frac{\partial q(\bfx; t)}{\partial t} \bigr \rangle_{\calP}.
    \end{equation}
    Observe that each barrier functional component in \eqref{sspf: variational_cbf} can be expressed as an inner product between a scaled indicator function and the variational density: $h_i(q(\bfx; t)) = - \left \langle g_i(\bfx) \bbI_{\calX \setminus \calS_i}(\bfx), q(\bfx, t) \right \rangle_{\calP}$. By the linearity of the inner product operator, we obtain:
    \begin{equation}
    \label{sspf: functional_derivative_variational_cbf}
        \frac{\delta h_i(q(\bfx; t))}{\delta q(\bfx; t)} = -g_i(\bfx) \bbI_{\calX \setminus \calS_i}(\bfx).
    \end{equation}
    Since the variational density evolution is governed by the Liouville equation, substituting \eqref{prelim: liouville} and \eqref{sspf: functional_derivative_variational_cbf} into \eqref{sspf: time_derivative_variational_cbf} yields: $\frac{\d h_i(q(\bfx; t))}{\d t} = \int_{\calX \setminus \calS_i} g_i(\bfx) \nabla_{\bfx} \cdot \big( q(\bfx; t) \phi(\bfx, t) \big) \d \bfx$.
    %
    %
    This expression can be simplified using Green’s theorem~\cite{Duistermaat_Kolk_2004}, yielding
    \begin{align*}
        \frac{\d h_i(q(\bfx; t))}{\d t} &= \oint_{\Gamma_i} g_i(\bfx) q(\bfx; t) \left \langle \bfphi(\bfx, t), \hat{\bfn}_i(\bfx) \right \rangle_{\bbR^n}  \d \bfx \\
        & \quad - \int_{\calX \setminus \calS_i} q(\bfx; t) \langle \nabla_{\bfx}g_i(\bfx), \bfphi(\bfx, t) \rangle_{\bbR^n} \d \bfx, 
    \end{align*}
    where $\Gamma_i \coloneqq \partial (\calX \setminus \calS_i)$ and $\hat{\bfn}_i(\bfx)$ is the outward unit normal vector to the boundary $\Gamma_i$. Since the boundary $\Gamma_i$ satisfies $\Gamma_i \subseteq \partial \calX \bigcup \partial \calS_i$, and the boundary of the $i$th constraint set satisfies $\partial \calS_i \subseteq \partial \calX \bigcup \{ \bfx \in \calX | g_i(\bfx) = 0 \}$. We have $\Gamma_i \subseteq \partial \calX \bigcup \partial  \{ \bfx \in \calX | g_i(\bfx) = 0 \}$. By Assumption~\ref{assum: mass_conservation}, we have $\langle \bfphi(\bfx,t), \hat{\bfn}_i(\bfx) \rangle_{\bbR^n} = 0$, for all $\bfx \in \partial \calX$. As a result, $g_i(\bfx) \left \langle \bfphi(\bfx, t), \hat{\bfn}_i(\bfx) \right \rangle_{\bbR^n} = 0$ for all $\bfx \in \Gamma_i$, which yields $\oint_{\Gamma_i} g_i(\bfx) q(\bfx; t) \left \langle \bfphi(\bfx, t), \hat{\bfn}_i(\bfx) \right \rangle_{\bbR^n}  \d \bfx = 0$. The condition in \eqref{sspf: dynamics_cbc} follows from the component-wise form of the constraint in \eqref{sspf: variational_cbc}. Forward-invariance of the feasible density set $\calP_s$ follows from \cite[Theorem~3]{kiss2023cbfal}, while exponential stability is established via Gr\"onwall's inequality~\cite{note1919gronwall}.
\end{proof}

Having established conditions under which the particle drift renders the feasible density set $\calP_s$ forward invariant and exponentially stable, we next tackle the construction of a particle drift satisfying these conditions.

\subsection{Particle Drift Design}
Inspired by CBF methods \cite{ames2017clfcbfqp}, we construct a safe particle drift $\bfphi(\bfx, t)$ by modifying the desired Stein particle drift $\bfphi_{d}(\bfx, t)$ in \eqref{sspf: stein_flow} that solves the unconstrained VI problem. We parameterize the particle drift $\bfphi(\bfx, t)$ as follows:
\begin{equation}
\label{sspf: parameterized_particle_dynamics}
    \bfphi(\bfx, t) = \bfphi_{d}(\bfx, t) + \bfu(\bfx, t), 
\end{equation}
where $\bfu: \calX \times [0, \infty) \to \bbR^n$ is an auxiliary control term introduced to modify the desired particle drift. The control term is chosen to satisfy the following condition for all $i \in \calI$:
\begin{equation}
\label{sspf: vector_cbc}
    \nabla_{\bfx} g_i(\bfx)^{\top} \left( \bfphi_d(\bfx, t) + \bfu(\bfx, t) \right) + \alpha_{g} g_i(\bfx) \geq 0, 
\end{equation}
where $\alpha_g > 0$ is a positive constant. An input satisfying \eqref{sspf: vector_cbc} can be obtained by solving the quadratic program: 
\begin{align}
    \bfu(\bfx&, t) = \argmin_{\bfu \in \bbR^n} \,\| \bfu \|_2^2  \label{sspf: cbf_qp_u} \\
    &\mathrm{s.t.} \, \, \nabla_{\bfx} g_i(\bfx)^{\top} \left( \bfphi_d(\bfx, t) + \bfu \right) + \alpha_{g} g_i(\bfx) \geq 0, \, \forall i \in \calI.\!
\end{align}
The optimization yields a minimally invasive control input that ensures satisfaction of the constraint \eqref{sspf: vector_cbc}. However, the input is well defined contingent upon the feasibility of \eqref{sspf: cbf_qp_u}. In this formulation, each constraint function $g_i(\bfx)$ can be considered as a CBF, and stacking them yields $\bfg(\bfx) = [g_1(\bfx), g_2(\bfx), \dots, g_N(\bfx)]^{\top}$. Feasibility of \eqref{sspf: cbf_qp_u} is guaranteed if $\bfg(\bfx)$ constitutes a valid vector CBF \cite{jorege2024flow}.
\begin{proposition}[Safe Particle Control]\label{prop:safe-particle-control}
    Let $\bfg(\bfx)$ a valid vector CBF \cite[Section~II.C]{jorege2024flow}, the particle drift $\bfphi(\bfx, t)$ defined by \eqref{sspf: parameterized_particle_dynamics}, with $\bfu(\bfx, t)$ obtained by solving \eqref{sspf: cbf_qp_u}, satisfies \eqref{sspf: dynamics_cbc} with $\alpha_h = \alpha_g$.
\end{proposition}
\begin{proof}
    The valid vector CBF condition guarantees the existence of a control input $\bfu(\bfx,t)$ such that, for each constraint function $g_i(\bfx)$, the particle drift satisfies $\langle \nabla_{\bfx}g_i(\bfx), \bfphi(\bfx, t) \rangle_{\bbR^n} \geq -\alpha_g g_i(\bfx)$ for all $\bfx \in \calX$. By the definition of PDFs, we can multiply both sides of the inequality by $q(\bfx; t)$ to obtain, for all $\bfx \in \calX$, $q(\bfx; t) \langle \nabla_{\bfx}g_i(\bfx), \bfphi(\bfx, t) \rangle_{\bbR^n} \geq -\alpha_g q(\bfx; t) g_i(\bfx)$.
    %
    %
    Integrating both sides over $\calX \setminus \calS_i$ preserves the inequality 
    %
    \begin{multline*}
       \int_{\calX \setminus \calS_i} q(\bfx; t) \langle \nabla_{\bfx}g_i(\bfx), \bfphi(\bfx, t) \rangle_{\bbR^n} \d \bfx \\
       \geq \alpha_g  \int_{\calX \setminus \calS_i} - q(\bfx; t) g_i(\bfx) \d \bfx. 
    \end{multline*}
    By the definition of the barrier functional \eqref{sspf: variational_cbf}, we have: $\int_{\calX \setminus \calS_i} q(\bfx; t) \langle \nabla_{\bfx}g_i(\bfx), \bfphi(\bfx, t) \rangle_{\bbR^n} \d \bfx \geq \alpha_g h_i(q(\bfx; t))$, 
    %
    %
    for all $i \in \calI$, which coincides with \eqref{sspf: dynamics_cbc} for $\alpha_h = \alpha_g$. 
\end{proof}

The key steps of the proposed method are summarized in Algorithm~\ref{alg:safe_pf}. Instead of the Stein particle drift \eqref{sspf: stein_flow}, our method can be formulated for other desired particle drifts, provided the associated gradient flow minimizes the KL divergence. Such drifts can be obtained by solving \eqref{sspf: particle_flow} with the KL divergence gradient computed under different Riemannian metrics \cite{chen2023gradient}. 

\begin{table}[t]
\centering
\caption{Runtime comparison of our safe PF, the projected PF \cite{craft2024nonlinear}, and CSVGD \cite{power2024constrained}. All methods employ the same ODE solver and integration time horizon.}\label{tab: runtime}
\begin{tabular}{|l|c|c|c|}
\hline
Method & Exec. Time (s) & Avg. It. Time (s) & Total It. \\
\hline
Safe PF (ours) & 9.922 & 0.060 & 164 \\
Projected PF & 9.558 & 0.043 & 218 \\
CSVGD & 176.293 & 0.187 & 941 \\
\hline
\end{tabular}
\end{table}

\begin{algorithm}[t]
\caption{Safe Particle Flow}
\label{alg:safe_pf}
\begin{algorithmic}[1]
\small
\Require Particles $\{\bfx_j(0)\}_{j=1}^{M}$, joint density $p(\bfx, \bfz)$, and observation $\tilde{\bfz}$
\Output Particles $\{\bfx_j(T)\}_{j=1}^{M}$ that approximate the solution to \eqref{prob: cons_bayes_infer_functional}

\Function{$\bff$}{$\{\bfx_j(t)\}_{j=1}^{M}$, $t$}
    \For{each particle \text{$\bfx_j(t)$}}
    \State $\bfphi_d(\bfx_j(t), t) \gets$ Evaluate \eqref{sspf: stein_flow} with $p(\bfx,\tilde{\bfz})$ and $\{\bfx_k(t)\}_{k=1}^{M}$ at $(\bfx_j(t),t)$
    \State $\bfu(\bfx_j(t), t) \gets$ Solve \eqref{sspf: cbf_qp_u} with $\bfphi_d(\bfx_j(t), t)$ at $(\bfx_j(t), t)$
    \State $\bfphi(\bfx_j(t), t) \gets \bfphi_d(\bfx_j(t), t) + \bfu(\bfx_j(t), t)$
    \EndFor
\State \Return $\{\bfphi(\bfx_j(t), t)\}_{j=1}^{M}$
\EndFunction

\While{ODE solver running}
    \State $\{\bfx_j(T)\}_{j=1}^{M} \gets$ SolveODE($\bff(\{\bfx_j(t)\}_{j=1}^{M}, t)$) with initialization $\left(\{\bfx_j(0)\}_{j=1}^{M}, t=0 \right)$ and termination time $T$
\EndWhile

\State \Return $\{\bfx_j(T)\}_{j=1}^{M}$ 
\end{algorithmic}
\end{algorithm}

\begin{figure*}[t]
\centering
\begin{minipage}[t]{0.6\linewidth}
    \centering
    \subcaptionbox{Projected PF~\cite{craft2024nonlinear} \label{fig: project_eq}}{\includegraphics[width=0.32\linewidth,trim=8mm 12mm 10mm 17mm, clip]{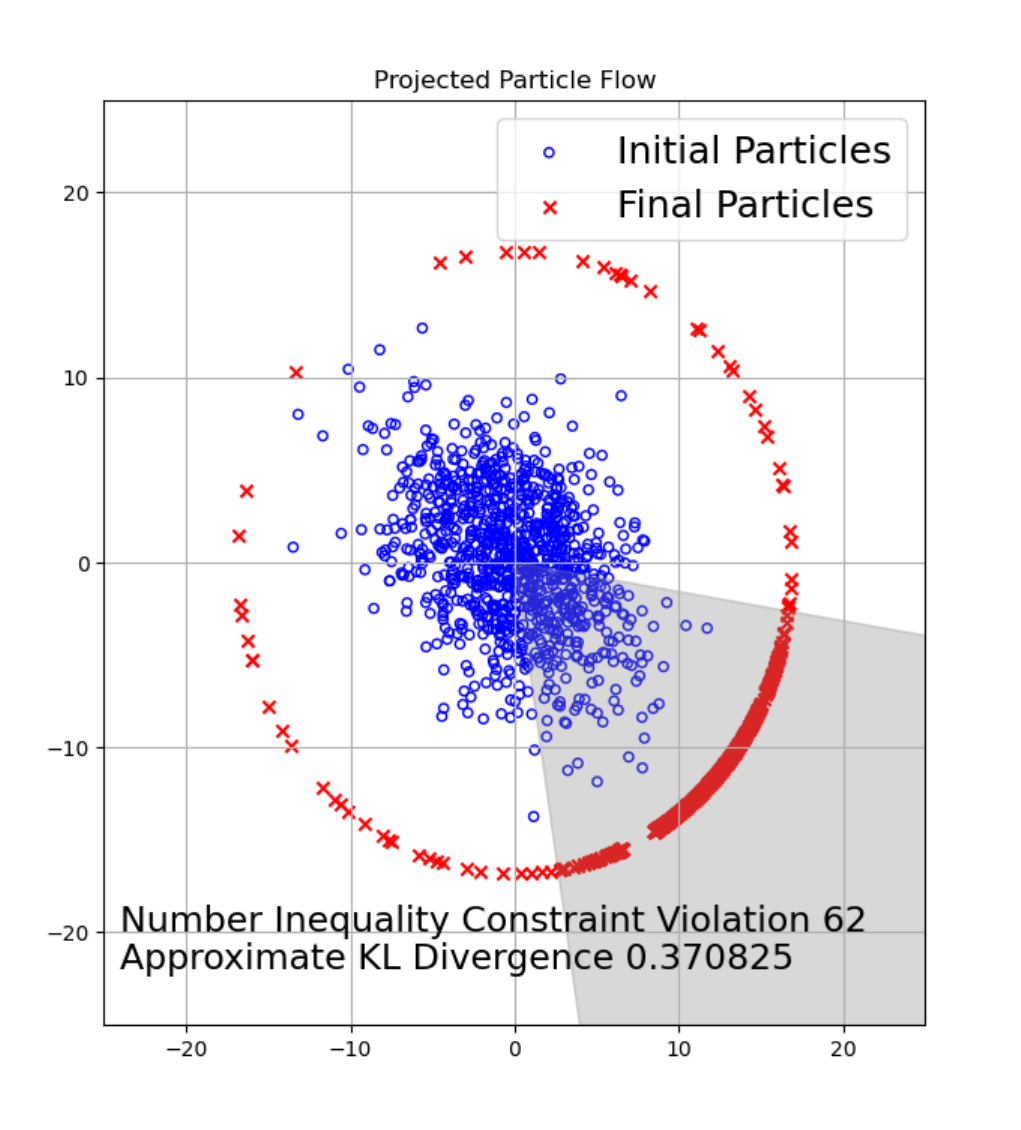}}%
    \hfill%
    \subcaptionbox{CSVGD \cite{power2024constrained} \label{fig: csvgd}}{\includegraphics[width=0.32\linewidth,trim=8mm 12mm 10mm 17mm, clip]{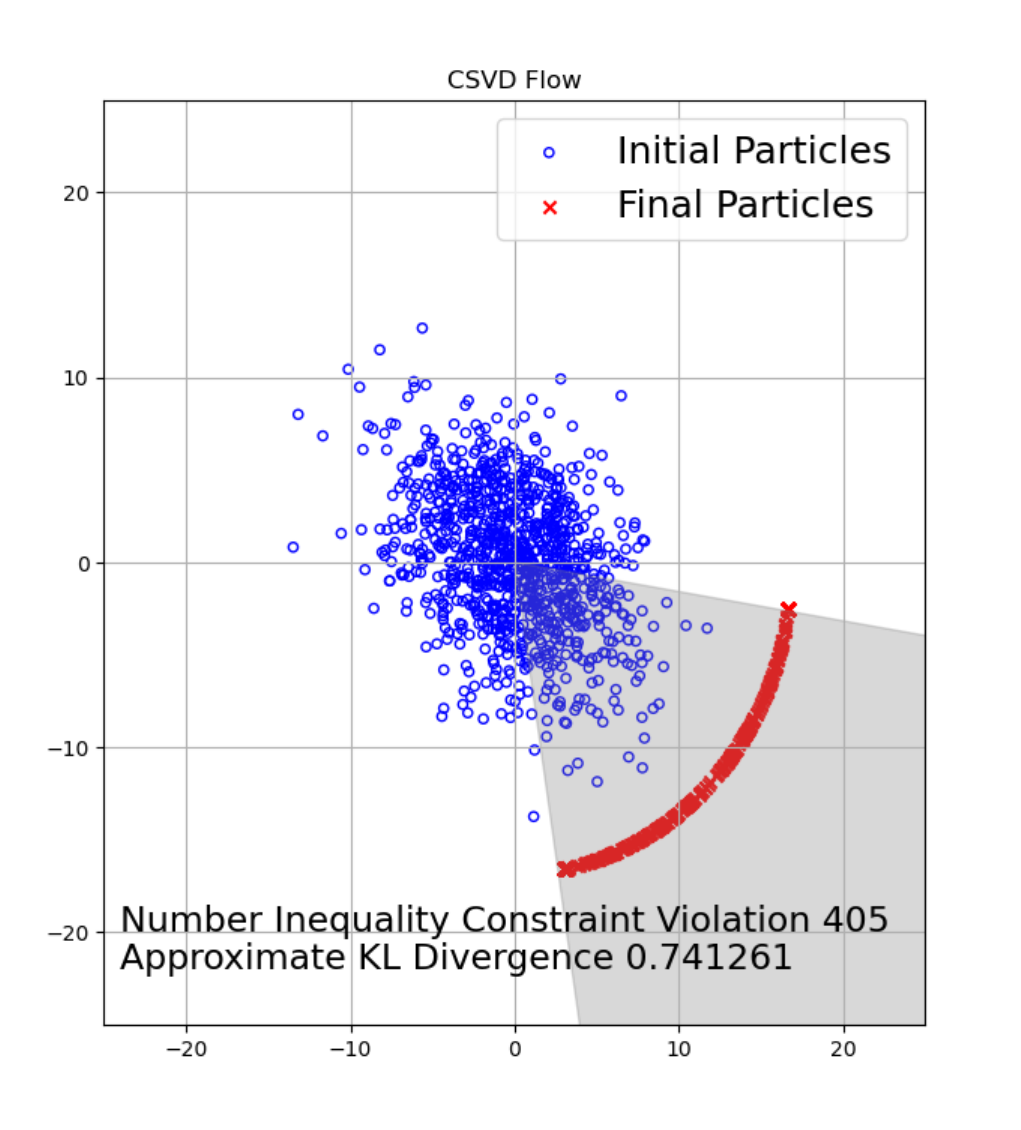}}%
    \hfill%
	\subcaptionbox{Safe PF (ours) \label{fig: safe_eq}}{\includegraphics[width=0.32\linewidth,trim=8mm 12.5mm 10mm 17mm, clip]{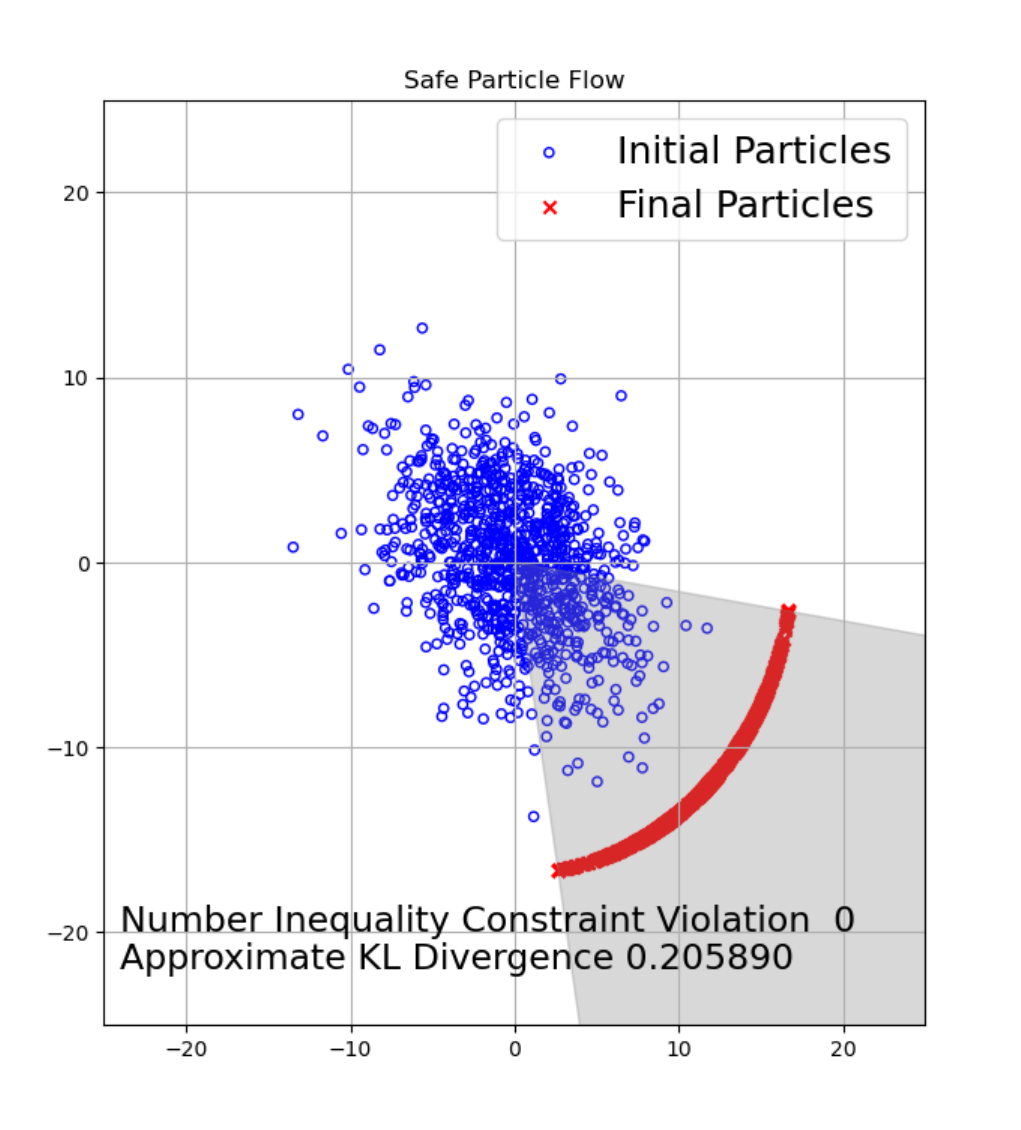}}%
    \caption{Comparison of our safe PF with CSVGD \cite{power2024constrained} and the projected PF \cite{craft2024nonlinear}. For each method, the desired particle drift is the Stein particle drift \eqref{sspf: stein_flow}. The region satisfying the inequality constraint in \eqref{eval: constraints} is shown in gray. The initial and final particles are shown as blue dots and red crosses, respectively.}
	\label{fig: compare_eq}
\end{minipage}%
\hfill%
\begin{minipage}[t]{0.38\linewidth}
    \centering
    \subcaptionbox{Projected PF~\cite{craft2024nonlinear} \label{fig: project_ineq}}{\includegraphics[width=0.5\linewidth,trim=8mm 12mm 10mm 17mm, clip]{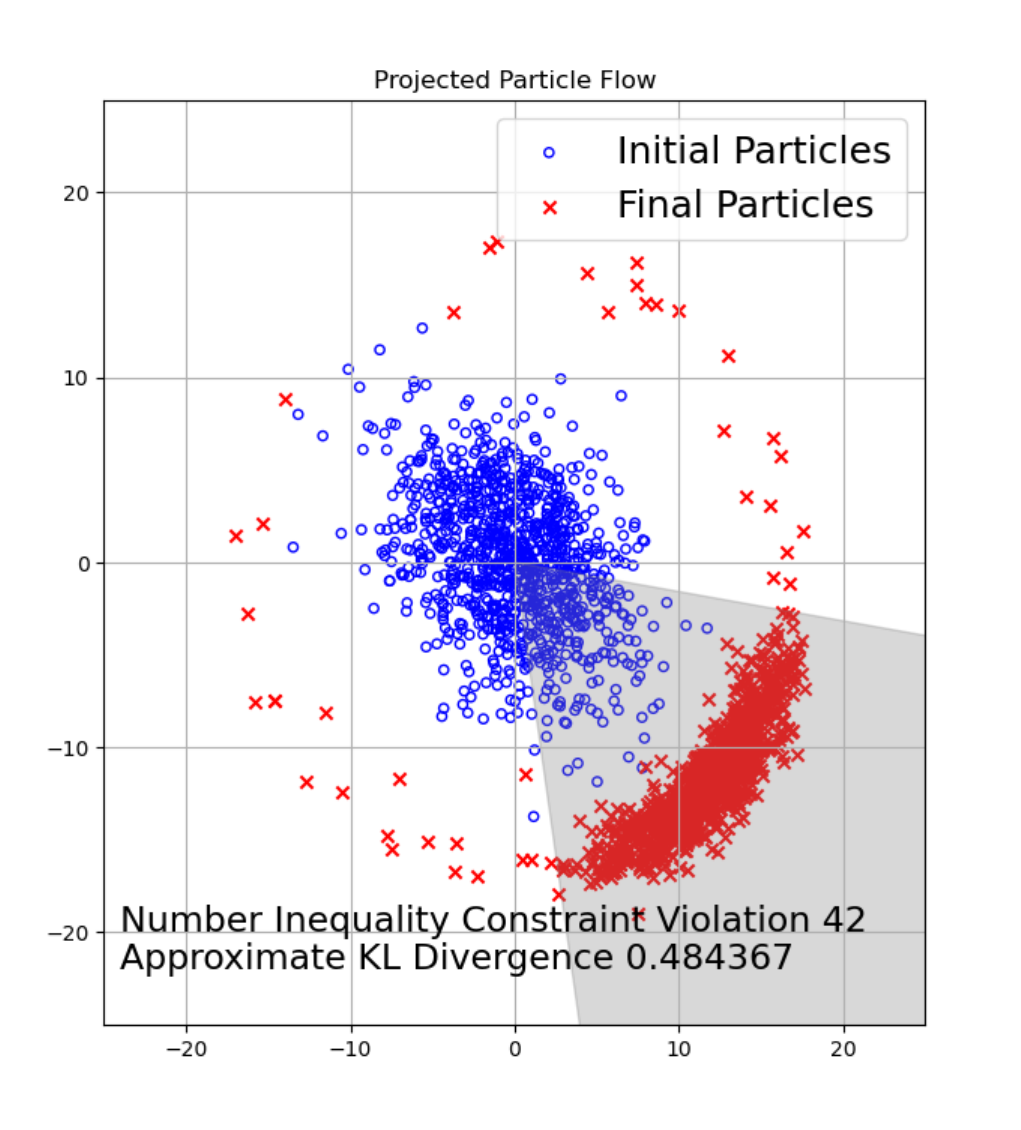}}%
    \hfill%
    \subcaptionbox{Safe PF (ours) \label{fig: safe_ineq}}{\includegraphics[width=0.5\linewidth,trim=8mm 12mm 10mm 17mm, clip]{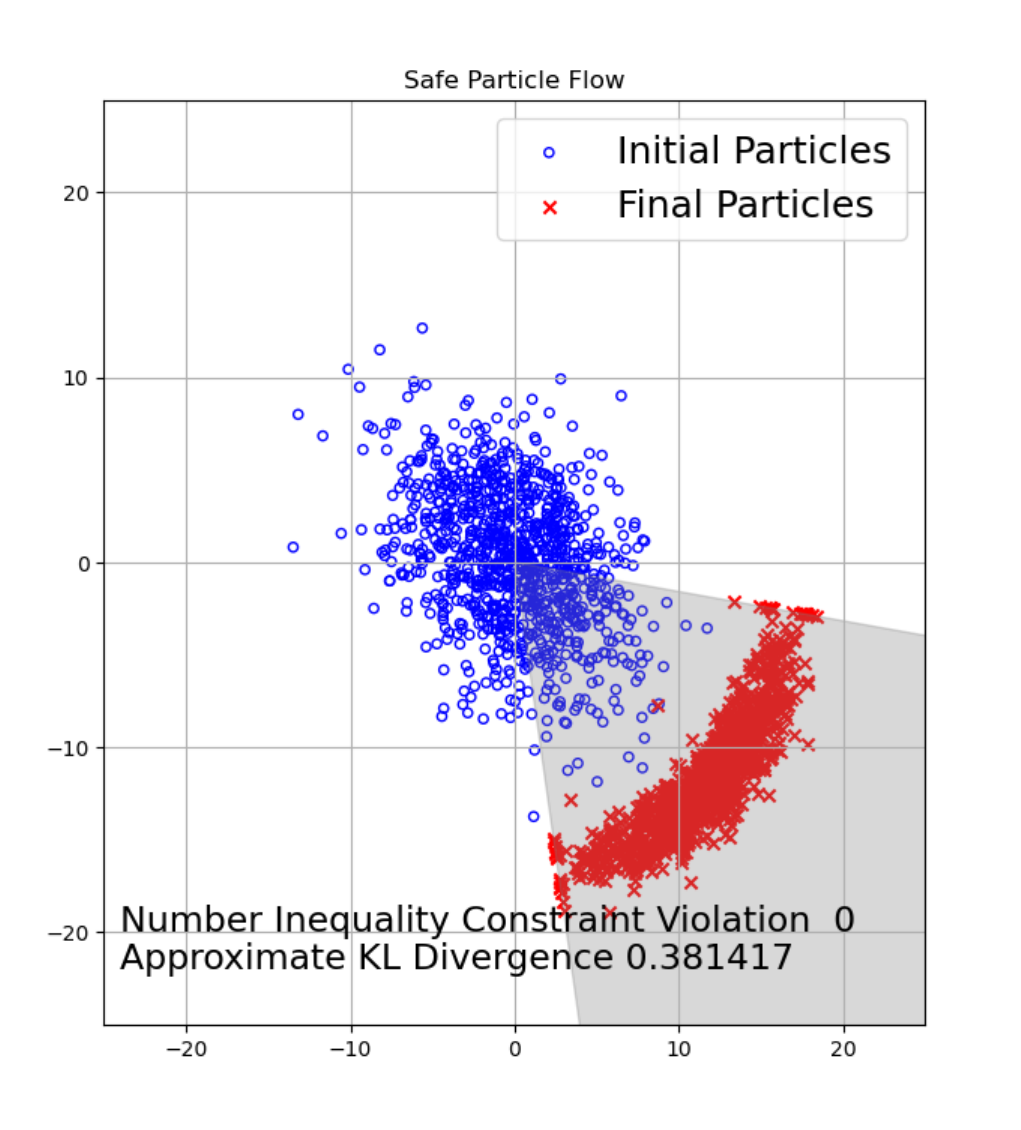}}%
    \caption{Comparison of safe PF with the projected PF \cite{craft2024nonlinear}, considering only the inequality constraint in \eqref{eval: constraints}. The initial and final particles are shown as blue dots and red crosses, respectively.}
    \label{fig: compare_ineq}
\end{minipage}
\end{figure*}

\section{Evaluation} 
\label{eva: sec_eva}

We evaluate the proposed safe particle flow method on a Bayesian estimation problem with one equality constraint and one inequality constraint. We compare the proposed method with the projection-based approaches \cite{craft2024nonlinear, power2024constrained}, and demonstrate that our method achieves better approximation accuracy while strictly satisfying the imposed constraints.

We consider the state space $\calX = \{ \bfx \in \bbR^2 \; | \; \| \bfx \|_{\infty} \leq 10^{3} \}$. The prior density is a truncated Gaussian supported on $\calX$ with $\hat{\bfx} = [0, 0]^{\top}$ and $P = \begin{bmatrix} 15 & -5 \\ -5 & 15 \end{bmatrix}$. The likelihood function is a Gaussian density $p(\bfz | \bfx) = p_{\calN}(\bfz; H(\bfx), R)$ with $H(\bfx) = \| \bfx \|$, $R = 1$, and $\bfx^* = [14.7, -10.1]^{\top}$, where $\bfx^*$ denotes the true value of $\bfx$ used to generate the observation. The following constraints are imposed on the Bayesian estimation problem:
\begin{equation} \label{eval: constraints}
\begin{split}
    g(\bfx) &= \pi / 5 - \arccos(\bfd^{\top} \bfx / \| \bfx \|_2) \geq 0 \\
    g^{e}(\bfx) &= \| \bfx \|^2 - r^2 = 0, 
\end{split}
\end{equation}
where $\bfd = \frac{1}{\sqrt{2}}[1, -1]^{\top}$ and $r=15.8$. The equality constraint enforces that the density is supported on a circle of radius $r$. The inequality constraint encodes a field-of-view restriction, requiring the density to lie within a cone centered along direction $\bfd$ with half-angle $\frac{\pi}{5}$. All methods are initialized with particles $\{\bfx_j(0)\}_{j=1}^{10^3}$ drawn from the prior. The desired particle drift is given by \eqref{sspf: stein_flow} with an RBF kernel of bandwidth $3$.

The results considering both equality and inequality are shown in Fig.~\ref{fig: safe_eq}. For reference, we also show in Fig.~\ref{fig: project_eq} the results obtained using the projected particle flow (PF) \cite{craft2024nonlinear} and in Fig.~\ref{fig: csvgd} the results of constrained Stein variational gradient descent (CSVGD) \cite{power2024constrained}. The projected PF satisfies the equality constraint but violates the inequality constraint. CSVGD achieves better inequality constraint satisfaction, as the particles violating the inequality constraint remain close to the safe region. Our safe PF satisfies both constraints and achieves a low KL divergence estimate. To exclude interference between equality and inequality constraints, we repeated the experiment with only the inequality constraint. The results are shown in Fig.~\ref{fig: compare_ineq}. The projected PF fails to satisfy the inequality constraint, as shown in Fig.~\ref{fig: project_ineq}. In contrast, our safe PF satisfies the inequality constraint while achieving good convergence, as shown in Fig.~\ref{fig: safe_ineq}. Our method has computational efficiency comparable to \cite{craft2024nonlinear}, while being significantly more efficient than \cite{power2024constrained}, as illustrated in Tab.~\ref{tab: runtime}.

\section{Conclusion}
We have introduced a novel safe particle flow method to satisfy constraints in VI problems. We have established that the constraints on the variational density can be equivalently reformulated as constraints on the particle drift. Combining ideas from safety control and the dynamical systems approach to algorithms, we have shown how to design a particle drift satisfying those constraints by solving a convex quadratic program. Our method proposed a simple yet efficient way to construct a safe particle flow while providing formal guarantees for constraint satisfaction for the variational density. Future work will focus on improving the efficiency of our method to enable real-time and high-dimensional applications, such as state estimation on manifolds and trajectory optimization.

    \bibliographystyle{cls/IEEEtran}
	\bibliography{ref/Safe_VI.bib}

@InProceedings{	  ames2019cbf,
  author	= {Ames, Aaron D. and Coogan, Samuel and Egerstedt, Magnus
		  and Notomista, Gennaro and Sreenath, Koushil and Tabuada,
		  Paulo},
  title		= {Control Barrier Functions: Theory and Applications},
  booktitle	= {European Control Conference (ECC)},
  year		= {2019},
  volume	= {},
  number	= {},
  pages		= {3420-3431},
  doi		= {10.23919/ECC.2019.8796030}
}

@Article{	  liu2016stein,
  title		= {{S}tein Variational Gradient Descent: A General Purpose
		  {B}ayesian Inference Algorithm},
  author	= {Liu, Qiang and Wang, Dilin},
  journal	= {{Advances in Neural Information Processing Systems}},
  volume	= {29},
  year		= {2016}
}

@Book{		  duistermaat_kolk_2004,
  title		= {Multidimensional Real Analysis II: Integration},
  publisher	= {Cambridge University Press},
  author	= {Duistermaat, J. J. and Kolk, J. A. C.},
  year		= {2004},
}

@Article{	  bayes1763,
  title		= {{LII.} {A}n {E}ssay towards {S}olving a {P}roblem in the
		  {D}octrine of {C}hances. {B}y the {L}ate {R}ev. {M}r.
		  {B}ayes, {F.R.S.} {C}ommunicated by {M}r. {P}rice, in a
		  {L}etter to {J}ohn {C}anton, {A.M.F.R.S.}},
  author	= {Bayes, Thomas},
  journal	= {Philosophical Transactions of the Royal Society of
		  London},
  volume	= {53},
  number	= {},
  pages		= {370--418},
  year		= {1763},
  publisher	= {The Royal Society London}
}

@InProceedings{	  wibisono2017information,
  title		= {Information and Estimation in {F}okker-{P}lanck
		  Channels},
  author	= {Wibisono, Andre and Jog, Varun and Loh, Po-Ling},
  booktitle	= {{IEEE} International Symposium on Information Theory},
  pages		= {2673--2677},
  year		= {2017}
}

@Book{		  bishop2006pattern,
  title		= {Pattern Recognition and Machine Learning},
  author	= {Bishop, Christopher M},
  year		= {2006},
  publisher	= {Springer}
}

@InProceedings{	  craft2024nonlinear,
  title		= {Nonlinear Particle Flow for Constrained {B}ayesian
		  Inference},
  author	= {Kyle J. Craft and Kyle J. DeMars},
  booktitle	= {AIAA Scitech Forum},
  pages		= {0428},
  year		= {2024}
}

@Article{	  kalman1960new,
  author	= {Kalman, Rudolph},
  title		= {A New Approach to Linear Filtering and
		  Prediction Problems},
  journal	= {{J}ournal of {B}asic {E}ngineering},
  volume	= {82},
  number	= {1},
  pages		= {35-45},
  year		= {1960}
}

@Book{		  anderson2005optimal,
  title		= {Optimal Filtering},
  author	= {Anderson, Brian DO and Moore, John B},
  year		= {2005},
  publisher	= {{C}ourier {C}orporation}
}

@Article{	  jordan1999introduction,
  title		= {An Introduction to Variational Methods for
		  Graphical Models},
  author	= {Jordan, Michael I and Ghahramani, Zoubin and Jaakkola,
		  Tommi S and Saul, Lawrence K},
  journal	= {{M}achine {L}earning},
  volume	= {37},
  pages		= {183--233},
  year		= {1999},
  publisher	= {{S}pringer}
}

@Article{	  khan2023bayesian,
  title		= {The {B}ayesian Learning Rule},
  author	= {Mohammad Emtiyaz Khan and H{\aa}vard Rue},
  journal	= {{J}ournal of {M}achine {L}earning {R}esearch},
  volume	= {24},
  number	= {281},
  pages		= {1--46},
  year		= {2023}
}

@InProceedings{	  geffner2023langevin,
  title		= {{L}angevin Diffusion Variational Inference},
  author	= {Geffner, Tomas and Domke, Justin},
  booktitle	= {International Conference on Artificial Intelligence and Statistics},
  pages		= {576--593},
  year		= {2023}
}

@article{ames2017clfcbfqp,
	author={Ames, Aaron D. and Xu, Xiangru and Grizzle, Jessy W. and Tabuada, Paulo},
	journal={IEEE Transactions on Automatic Control}, 
	title={Control Barrier Function Based Quadratic Programs for Safety Critical Systems}, 
	year={2017},
	volume={62},
	number={8},
	pages={3861-3876},
	doi={10.1109/TAC.2016.2638961}
}

@InProceedings{lee2019space,
author = {Taeyoung Lee},
title = {Spacecraft Attitude Estimation with a Single Magnetometer Using Matrix {F}isher Distributions on {SO}(3)},
booktitle = {AIAA Scitech Forum},
chapter = {},
pages = {},
doi = {10.2514/6.2019-1173},
year={2019},
}

@INPROCEEDINGS{heiden2022probabilistic,
  author={Heiden, Eric and Denniston, Christopher E. and Millard, David and Ramos, Fabio and Sukhatme, Gaurav S.},
  booktitle={International Conference on Robotics and Automation (ICRA)}, 

  title={Probabilistic Inference of Simulation Parameters via Parallel Differentiable Simulation}, 

  year={2022},

  pages={3638-3645}
}

@inproceedings{daum2007nonlinear,
  title={Nonlinear Filters with Log-Homotopy},
  author={Daum, Fred and Huang, Jim},
  booktitle={{SPIE} Signal and Data Processing of Small Targets},
  volume={6699},
  pages={423--437},
  year={2007},
}

@ARTICLE{jorege2024flow,
  author={Allibhoy, Ahmed and Cortés, Jorge},
  journal={IEEE Transactions on Automatic Control}, 
  title={Control-Barrier-Function-Based Design of Gradient Flows for Constrained Nonlinear Programming}, 
  year={2024},
  volume={69},
  number={6},
  pages={3499-3514},
  doi={10.1109/TAC.2023.3306492}
}

@inproceedings{gordon1993novel,
  title={Novel Approach to Nonlinear/non-{G}aussian {B}ayesian State Estimation},
  author={Gordon, Neil J and Salmond, David J and Smith, Adrian FM},
  booktitle={{IEE} {P}roceedings {F} {R}adar and {S}ignal {P}rocessing},
  volume={140},
  pages={107--113},
  year={1993},
}

@article{kiss2023cbfal,
author = {Kiss, Adam K. and Molnar, Tamas G. and Ames, Aaron D. and Orosz, Gabor},
title = {Control Barrier Functionals: Safety-Critical Control for Time Delay Systems},
journal = {{International Journal of Robust and Nonlinear Control}},
volume = {33},
number = {12},
pages = {7282-7309},
doi = {https://doi.org/10.1002/rnc.6751},
year = {2023}
}

@incollection{neal2011mcmc,
  title={{MCMC} Using {H}amiltonian Dynamics},
  author={Neal, Radford},
  booktitle={Handbook of Markov Chain Monte Carlo},
  pages={113--162},
  year={2011},
  publisher={CRC Press}
}

@article{chen2023gradient,
  title={Gradient Flows for Sampling: Mean-Field Models, {G}aussian Approximations and Affine Invariance},
  author={Chen, Yifan and Huang, Daniel Zhengyu and Huang, Jiaoyang and Reich, Sebastian and Stuart, Andrew M},
  journal={ar{X}iv preprint ar{X}iv:2302.11024},
  year={2023}
}

@inproceedings{prajna2004safety,
  title={Safety Verification of Hybrid Systems using Barrier Certificates},
  author={Prajna, Stephen and Jadbabaie, Ali},
  booktitle={International Workshop on Hybrid Systems: Computation and Control},
  pages={477--492},
  year={2004},
  organization={Springer}
}

@article{blanchini1999set_invariance,
title = {Set Invariance in Control},
journal = {Automatica},
volume = {35},
number = {11},
pages = {1747-1767},
year = {1999},
doi = {https://doi.org/10.1016/S0005-1098(99)00113-2},
author = {Franco Blanchini}
}

@book{gardiner2009stochastic,
  title={Stochastic Methods: A Handbook for the Natural and Social Sciences},
  author={Gardiner Crispin},
  year={2009},
  publisher={Springer Berlin Heidelberg}
}

@article{note1919gronwall,
 author = {Thomas Hakon Gronwall},
 journal = {Annals of Mathematics},
 number = {4},
 pages = {292--296},
 publisher = {[Annals of Mathematics, Trustees of Princeton University on Behalf of the Annals of Mathematics, Mathematics Department, Princeton University]},
 title = {Note on the Derivatives with Respect to a Parameter of the Solutions of a System of Differential Equations},
 volume = {20},
 year = {1919}
}

@book{barfoot2017state,
author = {Barfoot, Timothy D.},
title = {State Estimation for Robotics},
year = {2017},
publisher = {Cambridge University Press}
}

@article{mukadam2018continuous,
  title={Continuous-Time {G}aussian Process Motion Planning via Probabilistic Inference},
  author={Mukadam, Mustafa and Dong, Jing and Yan, Xinyan and Dellaert, Frank and Boots, Byron},
  journal={{The International Journal of Robotics Research}},
  volume={37},
  number={11},
  pages={1319--1340},
  year={2018},
  publisher={SAGE Publications Sage UK: London, England}
}

@ARTICLE{power2024constrained,
  author={Power, Thomas and Berenson, Dmitry},
  journal={{IEEE Transactions on Robotics}}, 
  title={Constrained {S}tein Variational Trajectory Optimization}, 
  year={2024},
  volume={40},
  number={},
  pages={3602-3619}
}

@article{tabor2025csvd,
  title={Constrained {S}tein Variational Gradient Descent for Robot Perception, Planning, and Identification}, 
  author={Griffin Tabor and Tucker Hermans},
  year={2025},
  journal={ar{X}iv preprint ar{X}iv:2506.00589}
}

@article{zhang2022sampling,
  title={Sampling in Constrained Domains with Orthogonal-Space Variational Gradient Descent},
  author={Zhang, Ruqi and Liu, Qiang and Tong, Xin},
  journal={Advances in Neural Information Processing Systems},
  pages={37\kern-0.05em108--37\kern-0.05em120},
  year={2022}
}

@article{gurbuzbalaban2024penalized,
  title={Penalized Overdamped and Underdamped {L}angevin {M}onte {C}arlo Algorithms for Constrained Sampling},
  author={Gurbuzbalaban, Mert and Hu, Yuanhan and Zhu, Lingjiong},
  journal={Journal of Machine Learning Research},
  volume={25},
  number={263},
  pages={1--67},
  year={2024}
}

@article{chamon2024constrained,
  title={Constrained Sampling with Primal-Dual {L}angevin {M}onte {C}arlo},
  author={Chamon, Luiz F and Karimi, Mohammad R and Korba, Anna},
  journal={{Advances in Neural Information Processing Systems}},
  pages={29\kern-0.05em285--29\kern-0.05em323},
  year={2024}
}
    
\end{document}